\documentclass[12pt,a4]{amsart} 
\usepackage{a4wide}
\usepackage{hyperref}

\usepackage[utf8]{inputenc} 
\usepackage{amsmath}
\usepackage{amssymb}
\usepackage{amstext}

\newtheorem{Lem}{Lemma}
\newtheorem{Thm}[Lem]{Theorem}
\newtheorem{Cor}[Lem]{Corollary}
\theoremstyle{definition}
\newtheorem{Def}[Lem]{Definition}

\newcommand{\FG}{\mathrm{FG}(\Omega)}

\begin{document}

\title{Free monoids are coherent}
\date{\today}

\author{Victoria Gould} 
\email{victoria.gould@york.ac.uk}
\author{Mikl\'{o}s Hartmann}
\email{miklos.hartmann@york.ac.uk}
\address{Department of Mathematics\\University
  of York\\Heslington\\York YO10 5DD\\UK}
  
\author{Nik Ru\v{s}kuc} 
\email{nik@mcs.st-and.ac.uk}
\address{School of Mathematics and Statistics\\
 University of St Andrews\\
 St Andrews, Fife KY16 9SS\\UK}

\begin{abstract}  A monoid $S$ is said to be {\em right coherent} if every finitely generated subact of every finitely presented right $S$-act is finitely presented. {\em Left coherency} is defined dually and $S$ is {\em coherent} if it is both right and left  coherent. These notions are analogous to those for a  ring $R$ (where, of course, $S$-acts are replaced by $R$-modules). Choo, Lam and Luft have shown that free  rings are coherent. In this note we prove that, correspondingly, any free monoid is coherent, thus answering a question posed by the first author in 1992.
\end{abstract}

\maketitle

\section{Introduction and preliminaries}\label{sec:ip}

The notion of right coherency for a monoid $S$ is defined in terms of finitary properties of right $S$-acts, corresponding to the  way in which right coherency is defined for a ring $R$ via properties of right $R$-modules. Namely,  $S$ is said to be {\em right (left) coherent} if every finitely generated subact of every finitely presented right (left) $S$-act is finitely presented.   If $S$ is both right and left  coherent then we say that  $S$ is {\em coherent}.
Chase \cite{chase:1960} gave equivalent internal conditions for right coherency of a ring $R$. The analogous result for monoids states that a monoid $S$ is right coherent if and only if for any finitely generated right congruence $\rho$ on $S$, and for any $a,b\in S$, the right annihilator congruence \[r(a\rho)=\{ (u,v)\in S\times S:au\,\rho\, av\}\] is finitely generated, and the subact $(a\rho)S\cap (b\rho)S$ of the right $S$-act $S/\rho$ is finitely generated (if non-empty) \cite{gould:1992}. {\em Left coherency} is defined for monoids and rings in a dual manner; a monoid or ring is {\em coherent} if it is both right and left coherent. Coherency is a rather weak finitary condition on rings and monoids and as demonstrated by Wheeler
\cite{wheeler:1976}, it is intimately related to the model theory of $R$-modules and $S$-acts. 

A natural question arises as to which of the important classes of infinite monoids are  (right) coherent? This study was initiated in \cite{gould:1992}, where it is shown that the free commutative monoid on any set $\Omega$ is coherent. For a (right) noetherian ring $R$, the free monoid ring $R[\Omega^*]$ over $R$ is (right) coherent \cite[Corollary  2.2]{choo:1972}. Since the free ring on $\Omega$ is the monoid ring $\mathbb{Z}[\Omega^*]$ \cite{mckenzie:1983}, it follows immediately that free  rings are coherent.  The question of whether the free monoid $\Omega^*$ itself is coherent was left open in \cite{gould:1992}. The purpose of this note is to provide a positive answer to that question:

\begin{Thm} \label{thm:main} For any set $\Omega$ the free monoid $\Omega^*$  is coherent. 
\end{Thm}

 Our proof of Theorem~\ref{thm:main}, given in Section~\ref{sec:result},  provides a blueprint for the proof in \cite{gould:future} that free left ample monoids are right coherent. Further comments are provided in Section~\ref{sec:comments}.

A few words on notation and technicalities follow. If $H$ is a set of pairs of elements of a monoid $S$, then we denote by
$\langle H\rangle$ the right congruence on $S$ generated by $H$. It is easy to see that if $a,b\in S$, then
$a\,\langle H\rangle\, b$ if and only if $a=b$ or there is an $n\geq 1$ and  a sequence
\[
(c_1,d_1,t_1;c_2,d_2,t_2;\dots; c_n,d_n,t_n)
\]
of elements of $S$, with $(c_i,d_i)\in H$ or $(d_i,c_i)\in H$, such that the following equalities hold:
\[a=c_1t_1,\, d_1t_1=c_2t_2,\, \hdots\, , d_nt_n=b.\]
Such a sequence will be referred to as an {\em $H$-sequence} (of length $n$) {\em connecting} $a$ and $b$. It is convenient to 
allow $n=0$ in the above sequence; the empty sequence is interpreted as asserting equality $a=b$. 
Where convenient we will use the fact that 
$\Omega^*$ is a submonoid of the free group $\FG$ on $\Omega$, in order to give the natural meaning to expressions such as
$yx^{-1}$, where $x,y\in \Omega^*$ and $x$ is a suffix of $y$. 

\section{Proof of Theorem~\ref{thm:main}}\label{sec:result}

Let $\Omega$ be a set; it is clearly enough to show that $\Omega^*$ is right coherent. To this end let $\rho$ be the right congruence on $\Omega^*$ generated by a finite subset $H$ of  $\Omega^* \times \Omega^*$, which without loss of generality we assume to be symmetric.

\begin{Def}\label{def:2}
A quadruple $(a,u;b,v)$ of elements of $S$ is said to be {\em irreducible}
if $(au,bv)\in \rho$ and for any common non-empty suffix $x$ of $u$ and $v$ we have that $(au x^{-1},bv x^{-1}) \not \in \rho$.
\end{Def}

\begin{Def}\label{def:1} 
An $H$-sequence
$(c_1,d_1,t_1;\dots;c_n,d_n,t_n)$ with
\[
au=c_1t_1,d_1t_1=c_2t_2,\ldots,d_nt_n=bv
\]
is {\em irreducible} with respect to $(a,u;b,v)$ if $u,t_1,\ldots,t_n,v\in \Omega^*$ do not have a common non-empty suffix.
Clearly, this is equivalent to one of $u,t_1,\ldots, t_n,v$ being $\epsilon$.
\end{Def}

 Throughout this note for an  $H$-sequence as above  we define $a=d_0,u=t_0,c_{n+1}=b$ and $t_{n+1}=v$. 
  It is clear that 
if the quadruple $(a,u;b,v)$ is irreducible then any $H$-sequence connecting $au$ and $bv$ must be irreducible with respect to 
$(a,u;b,v)$.

We define
\[
K=\text{max} \{ \left|p\right|: (p,q) \in H\}.
\]

\begin{Lem}\label{Tech}
Let the $H$-sequence $(c_1,d_1,t_1;\dots;c_n,d_n,t_n)$ with
\[
au=c_1t_1,d_1t_1=c_2t_2,\ldots,d_nt_n=bv
\]
be irreducible with respect to $(a,u;b,v)$.
Then either the empty $H$-sequence is irreducible with respect to $(a,u;c_1,t_1)$ (in which case $\left|u\right|\leq\text{max } (|b|, K)$ and $u=\epsilon$ or $t_1=\epsilon$) or there exist an index $1\leq i\leq n$ such that $t_{i+1}=\epsilon$ (so that $au\,\rho\, c_{i+1})$ and $x \in \Omega^+$ such that $\left|x\right|\leq\text{max }(\left|b\right|,K)$, the sequence
\[
(c_1,d_1,t_1x^{-1};\dots; c_{i-1},d_{i-1},t_{i-1}x^{-1})
\]
satisfies
\[
au x^{-1}=c_1 t_1 x^{-1},d_1t_1 x^{-1}=c_2t_2 x^{-1},\ldots,d_{i-1}t_{i-1} x^{-1}=c_it_i x^{-1},
\]
and is an irreducible $H$-sequence with respect to $(a,ux^{-1};c_i,t_ix^{-1})$. 
\end{Lem}

\begin{proof}
If the empty sequence is irreducible with respect to $(a,u;c_1,t_1)$ then either $u=\epsilon$ or $t_1=\epsilon$.
In both cases we have that $\left|u\right|\leq \text{max }(|b|,K)$.
Suppose therefore that the empty sequence is not irreducible with respect to $(a,u;c_1,t_1)$.
Let $i\in \{ 1,\hdots,  n\} $ be the smallest index such that $t_{i+1}=\epsilon$ (such an index exists, because our original sequence is irreducible), and let $x$ be the longest common non-empty suffix of $u=t_0,t_1,\ldots, t_i$.
Then the sequence
\[
(c_1,d_1,t_1x^{-1};\dots; c_{i-1},d_{i-1},t_{i-1}x^{-1})
\]
clearly satisfies
\[
au x^{-1}=c_1 t_1 x^{-1},d_1t_1 x^{-1}=c_2t_2 x^{-1},\ldots,d_{i-1}t_{i-1} x^{-1}=c_it_i x^{-1}
\]
and is irreducible with respect to $(a,ux^{-1};c_i,t_ix^{-1})$.
Furthermore, since $t_{i+1}=\epsilon$, we have that $d_it_i=c_{i+1}$, so $x$ is a suffix of $c_{i+1}$. If $i<n$ then
$(c_{i+1},d_{i+1})\in H$, while if $i=n$ we have $c_{i+1}=b$. In either case 
$|x|\leq |c_{i+1}|\leq \text{max }(|b|,K)$. 
\end{proof}

We deduce immediately that one condition for coherency of $\Omega^*$ is fulfilled. 

\begin{Cor}\label{cor:2}
Let $a,b \in S$.
Then $(a\rho)S \cap (b\rho)S$ is empty or finitely generated.
\end{Cor}

\begin{proof}
Let us suppose that $(a\rho)S \cap (b\rho)S\neq \emptyset$ and let
\[
X=\{a\rho,b\rho,c\rho: (c,d) \in H\}\cap (a\rho)S \cap (b\rho)S.
\]
We claim that $X$ generates $(a\rho)S \cap (b\rho)S$.
It is enough to show that for every irreducible quadruple $(a,u;b,v)$ we have that $(au)\rho \in X$.
For this, let $(c_1,d_1,t_1;\dots;c_n,d_n,t_n)$ be an $H$-sequence with
\[
au=c_1t_1,\ldots,d_nt_n=bv.
\]
Note that this sequence is necessarily irreducible with respect to $(a,u;b,v)$.
Then by Lemma \ref{Tech}, either $u=\epsilon$, or $t_i=\epsilon$ for some $i\in \{ 1,\hdots, n\}$, or $v=t_{n+1}=\epsilon$. In each of these cases we see that $(au)\rho \in X$.
\end{proof}

It remains to show that for any $a\in \Omega^*$, the right congruence $r(a\rho)$ is finitely generated. To this end we first present a technical result.

\begin{Lem}\label{Main}
Let $(c_1,d_1,t_1;\dots;c_n,d_n,t_n)$ with
$$
au=c_1t_1,\ldots,d_nt_n=bv
$$
be an irreducible $H$-sequence with respect to $(a,u;b,v)$.
Then either $u=\epsilon$, or there exist a factorisation $u=x_k\ldots x_1$ and indices $n+1\geq \ell_1 > \ell_2 > \ldots > \ell_k \geq 1$ such that for all $1 \leq j \leq k$: 

(i) $0<\left|x_j\right| \leq \text{max }(\left|b\right|,K)$  and 

(ii)  $au x_1^{-1} \ldots x_{j-1}^{-1} \mathrel{\rho} c_{\ell_j}$ (note that for $j=1$ we have $au \ \rho\ c_{\ell_1}$).
\end{Lem}

\begin{proof} We proceed by induction on $|u|$: if $|u|=0$ the result is clear. 
Suppose that $|u|>0$ and the result is true for all shorter words. 
If the empty sequence is irreducible with respect to $(a,u;c_1,t_1)$, then $t_1=\epsilon$ and the factorisation $u=x_1$ satisfies the required conditions, with $k=1$ and $\ell_1=1$.
Otherwise, by Lemma \ref{Tech}, there exist an index $1\leq i \leq n$ 
such that $t_{i+1}=\epsilon$, so that $au\,\rho\, c_{i+1}$, and $x_1 \in \Omega^+$ such that $\left|x_1\right|\leq \text{max }(\left|b\right|,K)$ and the sequence
\[
(c_1,d_1,t_1x_1^{-1};\dots; c_{i-1},d_{i-1},t_{i-1}x_1^{-1})
\]
satisfies
\[
au x_1^{-1}=c_1 t_1 x_1^{-1},d_1t_1 x_1^{-1}=c_2t_2 x_1^{-1},\ldots,d_{i-1}t_{i-1} x_1^{-1}=c_it_i x_1^{-1}
\]
and is an irreducible $H$-sequence with respect to $(a,ux_1^{-1};c_i,t_ix_1^{-1})$.  Put
$\ell_1=i+1$. Since $|ux^{-1}_1|<|u|$, the result follows by induction.
\end{proof}

\begin{Lem}\label{lem:1}
Let $a \in \Omega^*$.
Then $r(a\rho)$ is finitely generated.
\end{Lem}

\begin{proof}
Let $K'=\text{max}(K,\left|a\right|)+1, L=2\left|H\right|+2, N=K'L$ and define
\[
X=\{(u,v):\left|u\right|+\left|v\right|\leq 3N\} \cap r(a\rho).
\]
We claim that $X$ generates $r(a\rho)$.
It is clear that $\langle X\rangle \subseteq r(a\rho)$.


Let $(u,v) \in r(a\rho)$.
We show by induction on $\left|u\right|+\left|v\right|$ that $(u,v) \in \langle X \rangle$.
Clearly, if $\left|u\right|+\left|v\right|\leq 3N$, then $(u,v) \in X$.
We suppose therefore that  $\left|u\right|+\left|v\right|>3N$ and  make the inductive assumption that if $(u',v') \in r(a\rho)$ and $\left|u'\right| + \left|v'\right| < \left|u\right| + \left|v\right|$,  then $(u',v') \in \langle X\rangle$.
If the quadruple $(a,u;a,v)$ is not irreducible, it is immediate that  $(u,v) \in \langle X \rangle$.
Without loss of generality we therefore suppose that the quadruple
 $(a,u;a,v)$ is irreducible and $\left|v\right|\leq \left|u\right|$, so that $\left|u\right|>N$.
Let $(c_1,d_1,t_1;\dots;c_n,d_n,t_n)$ with
\[
au=c_1t_1,\ldots,d_nt_n=av
\]
be an irreducible $H$-sequence with respect to $(a,u;a,v)$. We apply Lemma~ \ref{Main}, noting here that $a=b$. Clearly $u\neq \epsilon$, so by
 Lemma \ref{Main}, there exists a factorisation $u=x_k\ldots x_1$ such that for all $1\leq j \leq k$ we have $0<\left|x_j\right|\leq K'$ and $au x_1^{-1}\ldots x_{j-1}^{-1} \mathrel{\rho} c_{\ell_j}$ for some $1\leq \ell_j \leq n+1$.
Since $\left|u\right| > K'L$ we have that $k >L$.
Note that the number of distinct  elements 
among $c_1,\hdots, c_n$  is less than $L-1$.
This in turn implies that there exist two indices $1\leq k-L < j<i \leq k$ such that $c_{\ell_i}=c_{\ell_j}$, so that
\[au x_1^{-1} \ldots x_{i-1}^{-1} \mathrel{\rho} c_{\ell_i}=c_{\ell_j} \mathrel{\rho} au x_1^{-1} \ldots x_{j-1}^{-1}.\]
Since  $i,j > k-L$ we have that $k-i+1 \leq L$, so $\left| u x_1^{-1}\ldots x_{i-1}^{-1}\right|=\left| x_k \ldots x_i\right| \leq K'L$, and similarly $\left|u x_1^{-1} \ldots x_{j-1}^{-1}\right| \leq K'L$.
As a consequence  $(u x_1^{-1}\ldots x_{i-1}^{-1},u x_1^{-1} \ldots x_{j-1}^{-1}) \in X$, and letting $u'=u x_1^{-1}\ldots x_{i-1}^{-1} x_{j-1} \ldots x_k$, we see that
\[
(u',u)=(u x_1^{-1}\ldots x_{i-1}^{-1},u x_1^{-1} \ldots x_{j-1}^{-1}) x_{j-1} \ldots x_1 \in \langle X\rangle.
\]
In particular, $au' \mathrel{\rho} au \mathrel{\rho} av$.
Note that $\left|u'\right|<\left|u\right|$, because $j<i$ and $x_j \neq \epsilon$.
Thus by the induction hypothesis we have that $(v,u') \in \langle X \rangle$ and so the lemma is proved.
\end{proof}

In view of the characterisation of coherency given in \cite{gould:1992} and  cited in the Introduction, Corollary~\ref{cor:2} and Lemma~\ref{lem:1} complete the proof of Theorem~\ref{thm:main}.

\section{Comments}\label{sec:comments}

Given that the class of right coherent monoids is closed under retract \cite{gould:future}, it  follows from the results of that paper that free monoids are coherent. However, as the arguments  in \cite{gould:future} for free left ample monoids are burdened with unavoidable technicalities, we prefer to present here the more transparent proof that $\Omega^*$ is coherent, by way of motivation for the work of \cite{gould:future}. With free objects in mind, we remark that we also show in  \cite{gould:future} that the  free inverse monoid 
on $\Omega$ is  not coherent  if $|\Omega|>1$.

\end{document}